\documentclass[11 pt,a4paper]{article}

\usepackage{amsmath,amsthm,amsfonts,amssymb,color}
  \definecolor{darkgreen}{rgb}{0,0.4,0}

\usepackage{url}

\usepackage{tikz-cd}
\DeclareMathOperator{\id}{id}

\parskip=\medskipamount
\def\lie#1{{\mathcal L}_{#1}}
\def\fpd#1#2{\frac{\partial #1}{\partial #2}}
\def\R{{\mathbb R}}
\def\SODE{{\textsc{sode}}}

\def\hook{{\mathchoice{\vrule height 0pt depth 0.4pt width 3pt
\vrule height 5pt depth 0.4pt \kern 3pt} {\vrule height 0pt depth
0.4pt width 3pt \vrule height 5pt depth 0.4pt \kern 3pt} {\vrule
height 0pt depth 0.2pt width 1.5pt \vrule height 3pt depth 0.2pt
width 0.2pt \kern 1pt} {\vrule height 0pt depth 0.2pt width 1.5pt
\vrule height 3pt depth 0.2pt width 0.2pt \kern 1pt} }}

\def\d{\mbox{d}}
\def\tGamma{{\tilde\Gamma}}
\def\det{\text{det}}

\theoremstyle{plain}
\newtheorem{thm}{Theorem}[section]
\newtheorem{lem}[thm]{Lemma}
\newtheorem{propn}[thm]{Proposition}
\newtheorem{cor}[thm]{Corollary}

\theoremstyle{definition}
\newtheorem{defn}[thm]{Definition}

\newtheorem{ex}[thm]{Example}

\newcommand{\pd}[2]{\frac{\partial#1}{\partial#2}}
\newcommand{\bnabla}{\bar{\nabla}}
\newcommand{\hnabla}{\hat{\nabla}}
\newcommand{\tnabla}{\tilde{\nabla}}

\newcommand{\Pnabla}{{\nabla}^P}
\newcommand{\Pbnabla}{{\bar\nabla}^P}
\newcommand{\tH}{\tilde{H}}
\newcommand{\tV}{\tilde{V}}
\newcommand{\tE}{\tilde{E}}
\newcommand{\tPhi}{\tilde{\Phi}}
\newcommand{\tpsi}{\tilde{\psi}}
\newcommand{\PH}{P_H}
\newcommand{\PV}{P_V}

\newcommand{\Sp}{\text{Sp}} 
\newcommand{\Img}{\text{\rm im}}
\newcommand{\X}{\mathfrak{X}}
\newcommand{\F}{\mathfrak{F}}
\newcommand{\M}{\mathcal{M}}

\begin{document}

\title{Linear connections and shape maps for second order ODEs\\ with and without constraints}

\author{G.E.\,Prince, M.\,Farr\'{e} Puiggal\'{i}, D.J.\,Saunders, D.\,Mart\'{i}n de Diego}


\maketitle

\begin{quote}
{\bf Abstract.} {\small  We deal with the construction of linear connections associated with second order ordinary differential equations with and without first order constraints. We use a novel method allowing glueing of submodule covariant derivatives to produce new, closed form expressions for the Massa-Pagani connection and our extension of it to the constrained case.\\
{\em Subject Classification (2020):} {Primary 34A26, 53C05, 70G45 Secondary 34C40, 53B05}\\
{\em Keywords:} {Linear connections, covariant derivatives, Massa-Pagani connection, second order ordinary differential equations, nonholonomic mechanics}}
\end{quote}

\section{Introduction}
The purpose of this paper is to construct and examine specific linear connections for general systems of second order differential equations (\SODE s) in both the absence and presence of first order constraints. While there is significant achievement in the former, unconstrained case (e.g., see \cite{CMS96,JP01,MaPa94,MS01,SM00}, the constrained case has proven problematic (see \cite{SCS95,SCS97,SSC96}). We believe that in this paper we have successfully produced a canonical linear connection for constrained systems using a new construction in the unconstrained case and the underlying geometry on the base constraint manifold. The geometric setting in the unconstrained case will be that of the evolution space of a base configuration manifold in whose coordinates the \SODE s are expressed. In the constrained case we take the cartesian product of this configuration space with the base constraint manifold.

We refer the reader to the recent papers by Mart\'{\i}nez \cite{EM18,EM19} for the historical background and geometric construction of the linearisation of nonlinear connections on vector bundles. While we do give a geometric, jet bundle description of our constrained connection, overall our approach uses covariant derivatives, relying on an apparently new result which allows covariant derivatives on distributions or submodules to be glued together. This method allows us to deal directly with the horizontal and vertical distributions (on the evolution space) provided by  the well-known {\em nonlinear} connection for unconstrained \SODE s. Each of these distributions comes with a natural covariant derivative and these may be extended and glued together to give a covariant derivative/linear connection on the whole evolution space. The construction in the constrained case simply requires an additional submodule covariant derivative.

This approach gives a straightforward formula for the covariant derivative of the well-known linear connection of Massa-Pagani and for our extension of it to the constrained case. Massa and Pagani introduced their connection by identifying it as the unique connection with certain natural geometric properties; while they computed the component-wise covariant derivatives they did not give an intrinsic formula for the covariant derivative itself.  Until now the formulae for the Massa-Pagani connection and other related Berwald connections were defined on a certain pull-back bundle and then lifted to the evolution space. Our approach gives the ``missing" intrinsic formulae on the evolution space without recourse to pull-back bundle constructs. The modification to the constrained case is then straightforward.

Our main motivation for studying linear connections associated to a SODE  subjected to first order constraints comes from nonholonomic mechanics \cite{NeiFuf72, Bloch03, Cortes2002}. Nonholonomic systems are dynamical systems on the tangent bundle of a manifold with constraints in the velocities that are usually defined by a non-integrable distribution. This non-integrability  property implies that the constraints do not impose restrictions on the possible configurations of the system, in contrast to holonomic constraints.

Examples of nonholonomic systems typically appear in rolling motion without slipping being widely used in robotic applications. Moreover, one of the aspects that makes nonholonomic dynamics very interesting  is that the equations are derived using the non-variational Lagrange-d'Alembert principle \cite{Arnold}.
This opens the question of  trying to characterise when a given nonholonomic system admits  a purely Lagrangian or Hamiltonian formulation in the manner of the classical {\em inverse problem in the calculus of variations}, see \cite{D41, KP08}.  A recent approach consists in trying to view the trajectories of the nonholonomic system as the restriction to the constraint submanifold of the trajectories of a Lagrangian system, that is, the Euler-Lagrange equations of a Lagrangian gives us the nonholonomic equations when we restrict the initial conditions to the constraint submanifold  \cite{BFMM15, BlFeMe09}, other approaches include Chaplygin Hamiltonisation \cite{Garcia-Naranjo16}.
	
One of the most successful  breakthroughs  in the study the inverse problem in calculus of variations was made  by applying the method of exterior differential systems \cite{AnTh92} and an important object for subsequent work was precisely  the  connection of Massa and Pagani \cite{MaPa94} (see also \cite{DP16,DP21}). Therefore, we believe that the linear connection we are defining here for a constrained SODE is the first step in the geometrical study the inverse problem for nonholonomic dynamics.

This paper is structured as follows. In the next section we identify how to construct a covariant derivative from component submodule derivatives and use the result to describe the torsion, curvature and shape map of such a connection. In the third section we turn to \SODE s without constraints, firstly setting up the geometric structure for their description and then demonstrating the explicit formula for the Massa-Pagani connection using the construction in section 2. We also explicitly show that our covariant derivative is the unique such derivative satisfying Massa and Pagani's original geometric conditions. We also compute the torsion and shape map for the connection.

We turn to constrained systems of \SODE s in section four beginning with the geometrical setup. Then we give a jet bundle description of the construction of a linear connection in the presence of constraints. We follow this with the explicit formula for such a linear connection using the results of sections two and three, giving a natural and transparent generalisation of the unconstrained case. Again we identify natural geometric conditions which uniquely specify the derivative {\em à la} Massa and Pagani.
We then demonstrate how our constrained equations arise in the non-holonomic context. In the final section we give a number of examples, focussing on constrained systems from non-holonomic mechanics.

\section{The construction of covariant derivatives}
We present an apparently new construction of a covariant derivative on a manifold from covariant derivatives on (vector) distributions or submodules. We begin by demonstrating an extension of these submodule derivatives. (In the remainder of the paper these derivatives are generically referred to as `submodule derivatives' even if the context is that of distributions.)

Let $\M$ be a manifold and $P$ a projector of the module of vector fields, $\X(\M)$, onto a given distribution or submodule.
Suppose that $\Img(P)$ has a (non-trivial) covariant derivative $\Pnabla$, that is $\Pnabla_XY$ has the usual properties for $X,Y\in \Img(P)$ over the smooth functions, $\mathfrak{F}(\M),$ on $M$, namely $\R$-linearity in both arguments and
\begin{equation}\label{cov deriv props}
\Pnabla_{fX}Y=f\Pnabla_XY \ \text{and}\ \Pnabla_X(fY)=X(f)Y+f\Pnabla_XY
\end{equation}
Notes: It is not required that $\Pnabla_XY \in \Img(P)$). While any projector onto the fixed submodule could be used to label the covariant derivative, in what follows $P$ must also be fixed.

$\Pnabla$ can be extended as follows:

\begin{propn}\label{propn1}
Define $\Pbnabla_XY$ for $X\in\X(\M), Y\in \Img(P)$ over $\F(\M)$ by
\begin{equation}
\Pbnabla_XY:=\Pnabla_{P(X)}Y+P([X-P(X),Y]).
\end{equation}
Then $\Pbnabla_XY$ has the covariant derivative properties.
\end{propn}

\begin{proof} The $\R$-linearity properties are obvious.
\begin{align*}
\Pbnabla_{fX}Y&=\Pnabla_{P(fX)}Y+P([fX-P(fX),Y])\\
&=f(\Pnabla_{P(X)}Y+P([X-P(X),Y]))-Y(f)P(X-P(X))\\
&=f\Pbnabla_XY \quad\text{since}\ P^2=P.
\end{align*}
\begin{align*}
\Pbnabla_X(fY)&=\Pnabla_{P(X)}(fY)+P([X-P(X),fY])\\
&=P(X)(f)Y+f\Pnabla_{P(X)}Y+fP([X-P(X),Y])+(X-P(X))(f)P(Y)\\
&=X(f)Y+f\Pbnabla_XY \quad \text{since}\ Y\in \Img(P).
\end{align*}
\end{proof}
Now suppose that $T\M$ has a direct sum decomposition into $N$ submodules or distributions with projectors $P_A, A=1,\dots,N.$
If each such submodule has a (non-trivial) covariant derivative $\nabla^A$ with extension ${\bar\nabla}^A$ as defined in proposition \ref{propn1} then
\begin{thm} \label{thm1} For $X,Y\in\X(\M)$
\[\nabla_XY:=\sum_{A=1}^N {\bar\nabla}^A_X(P_A(Y))\]
has the covariant derivative properties over $\F(\M).$
\end{thm}

\begin{proof}  Again, $\R$-linearity is obvious.
\begin{align*}
\nabla_{fX}Y&:=\sum_{A=1}^N {\bar\nabla}^A_{fX}(P_A(Y))\\
&=\sum_{A=1}^N f{\bar\nabla}^A_X(P_A(Y))=f\nabla_XY.
\end{align*}
\begin{align*}
\nabla_{X}(fY)&:=\sum_{A=1}^N {\bar\nabla}^A_X(P_A(fY))\\
&=\sum_{A=1}^N (f{\bar\nabla}^A_X(P_A(Y))+ X(f)P_A(Y))\\
&=f\nabla_XY+X(f)Y \quad \text{since}\ \sum_{A=1}^NP_A=I_{T\M}.
\end{align*}

\end{proof}
The following corollary will be useful in the direct sum situation of theorem \ref{thm1}.

\begin{cor}\label{cor2.3}
$\nabla P_B=0$ for all $B\in \{1,\dots,N\}$ if and only if $\nabla^B_XY\in\Img(P_B)$ for all $X,Y\in\Img(P_B)$ and for each $B\in\{1,\dots,N\}.$
\end{cor}

\begin{proof}
\begin{align*}
(\nabla_XP_B)(Y)&=\nabla_X(P_B(Y))-P_B(\nabla_XY)\\
&=\sum^N_{A=1}\bnabla^A_X(P_A(P_B(Y)))-P_B\left(\sum^N_{A=1}\bnabla^A_X(P_A(Y))\right)\\
&=\bnabla^B_X(P_B(Y))-P_B\left(\sum^N_{A=1}\left(\nabla^A_{P_A(X)}(P_A(Y))+P_A([X-P_A(X),P_A(Y)])\right)\right)\\
&=\bnabla^B_X(P_B(Y))-P_B\left(\sum^N_{A=1}\nabla^A_{P_A(X)}(P_A(Y))\right)-P_B([X-P_B(X),P_B(Y)])\\
&=\nabla^B_{P_B(X)}(P_B(Y))-P_B\left(\sum^N_{A=1}\nabla^A_{P_A(X)}(P_A(Y))\right)
\end{align*}

So, if $\nabla^A_XY \in \Img(P_A)$ for all $X,Y\in \Img(P_A)$ then the right-hand side is zero. Conversely, if $(\nabla_XP_B)(Y)=0$ for all $B,X,Y$ then for fixed $C$ and for $X,Y\in\Img(P_C)$ we have for all $B\neq C$
\[
0=-P_B(\nabla^C_{P_C(X)}(P_C(Y)))=-P_B(\nabla^C_XY).
\]
Hence $\nabla^C_XY\in\Img(P_C)$  for all $X,Y\in \Img(P_C),$ recalling that $I_{T\M}=\sum_{A=1}^N P_A.$
\end{proof}

Proposition \ref{propn1} and theorem \ref{thm1} provide the building blocks for the covariant derivatives we will construct in the remaining sections.

With any linear connection $\nabla$  on a manifold $\M$ there is an associated shape map and torsion (see \cite{JP01}).
The shape map $A_Z: \X(\M) \rightarrow \X(\M)$, as given in Jerie and Prince~\cite{JP01}, is the (1,1) tensor
$$
A_Z(\xi) := \frac{d}{dt}\Bigl|_{t=0} \tau_t^{-1}(\zeta_{t*}\xi), \qquad \text{where }\xi \in T_x\M,
$$
where $\tau_t: T_x\M \rightarrow T_{\zeta_t(x)}\M$ is the parallel transport map, defined along the flow $\{\zeta_t\}$ of $Z$.\\
Useful representations of this shape map on $\M$ are
\begin{align}
A_X (Y) &= \nabla_X Y - [X,Y], \label{A-Gamma-defn}\\
A_X(Y) &= \nabla_YX +  T(X,Y) ,\label{Torsion-shape map}
\end{align}
where $X,Y \in \X(\M)$  (in this context see also Kobayashi and Nomizu, Volume 1, p235 \cite{KN63}). The torsion and curvature are defined respectively as usual by
\begin{align*}
&T(X,Y):= \nabla_XY-\nabla_YX-[X,Y],\\
&R(X,Y)Z:=\nabla_X\nabla_YZ-\nabla_Y\nabla_XZ-\nabla_{[X,Y]}Z.
\end{align*}

In the direct sum situation of theorem \ref{thm1} there are only two interesting cases.
\begin{cor}\label{cor1}

\begin{itemize}
\item[]
\item[]
If $X\in\Img(P_B), Y\in\Img(P_C), B\neq C,$ then
\[\nabla_XY=P_C([X,Y]),\quad T(X,Y)=-\sum_{D\neq B,C}P_D([X,Y]),\quad A_X(Y)=-\sum_{D\neq C}P_D([X,Y]) \]
\item[]
If $X\in\Img(P_B), Y\in\Img(P_B)$ then
\[\nabla_XY=\nabla^B_XY,\quad T(X,Y)=T^B(X,Y),\quad A_X(Y)=A^B_X(Y),\]
\end{itemize}
where $T,A$ and $T^B,A^B$ are the torsions and shape maps of $\nabla$ and $\nabla^B$ respectively.
\end{cor}
(The results for the curvature are straightforward but involve more cases.)

\section{Unconstrained SODEs}
We now briefly describe the basics of our geometrical calculus, for more details we refer to the book chapter \cite{KP08}. In this section we are dealing with a system of smooth second-order ordinary differential equations
\begin{align}\label{SODE1}
\ddot x^i = F^i (t, x^j, \dot x^j), \ \ i, j = 1, \dots, n,
\end{align}
on a manifold $M$ (not to be confused with $\M$ in the last section) with generic local coordinates $(x^i)$ and with
associated bundles $\pi:\R \times M \rightarrow M,$ $t:\R \times M \rightarrow
\R$ and $\pi^0_1:E \rightarrow \R \times M$. The {\em evolution space} $E:=\R \times TM$ has adapted coordinates $(t,x^i,\dot x^i)$ or $(t,x^i,u^i)$. It is useful to identify $E$ with $J^1(\R,M)$, the bundle of 1-jets of maps $\R\to M$. $E$ has two natural structures: its {\em contact system} and its {\em vertical sub-bundle}.

The contact (co-)distribution on $J^1(\R,M)$ is 
$$
\Omega^1(\R,M)=\Sp\{\theta^i:=dx^i-u^idt\}.
$$
The integral submanifolds of $\Omega^1(\R,M)$ contain the lifts of graphs (1-graphs) of functions $f:\R\to M$.

The vertical sub-bundle $V(E)$ consists of the vertical subspaces of the tangent spaces of $E$, at each point $p\in E$ this is the kernel of ${(\pi^1_0)}_{\ast}:T_pE\to T_{\pi^0_1(p)}(\R\times M).$ The coordinate basis for $V(E)$ is $\{V_i:=\frac{\partial}{\partial u^i}\}$ and this generates the {\em vertical distribution} on $E$.

Finally, these two structures are combined in the {\em vertical endomorphism}
$$ S:=\theta^i\otimes V_i.$$
An intrinsic definition of $S$ can be found in \cite{CPT84}.

Now we identify equations \eqref{SODE1} with the codimension $n$  embedded submanifold of $J^2(\R,M):$
$$
\mathcal F:=\{q\in J^2(\R,M):v^i(q)-F^i(\pi^1_2(q))=0\},
$$
the coordinates on $J^2(\R,M)$ being $(t,x^i,u^i,v^i)$.
 The contact distribution on $J^2(\R,M)$ is
$$
\Omega^2(\R,M)=\Sp\{dx^i-u^idt,du^j-v^jdt\}
$$
and, if $i_{\mathcal F}$ is the inclusion map $\mathcal F \hookrightarrow J^2(\R,M)$, then the contact distribution restricted to $\mathcal F$ is
$$
i_{\mathcal F}^\ast \Omega^2(\R,M)=\Sp\{\theta^i:=dx^i-u^idt, \phi^j:=du^j-F^jdt\}
$$
with annihilator generated by the  {\em second-order differential equation field} (\SODE),
\begin{equation}\label{SODE2}
\Gamma := \frac{\partial}{\partial t} + u^i
\frac{\partial}{\partial x^i} + F^j \frac{\partial}{\partial u^j}.
\end{equation}
Because $i_{\mathcal F}$ can be viewed as a section of $\pi^1_2: J^2(\R,M) \to J^1(\R,M)$ we can define both the restricted contact distribution and $\Gamma$ on $E.$ The geometric interpretation being that the integral curves of $\Gamma$ are the lifted graphs of solution curves to \eqref{SODE1}, as evidenced by the conditions $\theta^i(\Gamma)=0=\phi^i(\Gamma)$.

 It is shown in \cite{CPT84} that the first order deformation $\mathcal{ L}_\Gamma S$ has eigenvalues $0, 1$ and $-1$. The eigenspaces at each point of $E$ corresponding to eigenvalues $0, 1$ and $-1$ are $\Sp\{\Gamma\}$, the {\em vertical distribution} $V(E):=\Sp\{V_i\}$ of the tangent space and the {\em horizontal distribution} $H(E):=\Sp\{H_i\}$ respectively, where
 $$
H_i := \frac{\partial}{\partial x^i} -\Gamma^j_i\frac{\partial}{\partial u^j},\ \Gamma_j^i := -{\frac{1}{2}}\frac{\partial F^i}{\partial u^j}.
$$
The resulting direct sum decomposition of the tangent spaces of $E$ gives an adapted local basis $\{\Gamma, V_i, H_i\}$ with dual basis $\{dt,\psi^i,\theta^i\}$ where
 $$
 \psi^i:=\phi^i+\Gamma^i_j\theta^j.
 $$
 The corresponding projectors  are denoted ${P_\Gamma, P_V} $ and $P_H$.

 In this way the \SODE \  produces a {\em nonlinear connection} on $E,$  with components $\Gamma_j^i$ appearing as follows:

\[
[\Gamma, H_i]=\Gamma^j_iH_j + \Phi^j_iV_j, \quad [\Gamma,
V_i]=-H_i + \Gamma^j_iV_j, \quad [V_i,V_j]=0,
\]
\[
[H_i, V_j]=-\frac{1}{2}\left(\frac{\partial^2 F^k}{\partial u^i\partial
u^j}\right)V_k=V_j(\Gamma^k_i)V_k=V_i(\Gamma^k_j)V_k=[H_j, V_i],
\]
and
\[
[H_i,H_j]=R^k_{ij} V_k,
\]
where
\[
\Gamma_j^i := -{\frac{1}{2}}\frac{\partial F^i}{\partial u^j}, \quad \Phi_j^i := -\frac{\partial F^i}{\partial x^j} -
\Gamma_j^k\Gamma_k^i - \Gamma(\Gamma_j^i).
\]
The {\em Jacobi endomorphism} is  $\Phi=P_V\circ\lie{\Gamma}P_H =\Phi^i_j\theta^j\otimes V_i,$
and the curvature of the connection is defined by
$$
R^k_{ij}:=\frac{1}{2}\left (\frac{\partial^2 F^k}{\partial
x^i\partial u^j} - \frac{\partial^2 F^k}{\partial x^j\partial u^i}
+\frac{1}{2}\left (\frac{\partial F^l}{\partial
u^i}\frac{\partial^2 F^k}{\partial u^l\partial u^j} -
\frac{\partial F^l}{\partial u^j}\frac{\partial^2 F^k}{\partial
u^l\partial u^i} \right )\right ).
$$
In our chosen basis the curvature tensor is
$$R=R^k_{ij}\theta^i\wedge \theta^j\otimes V_k.$$
For completeness' sake we give the important relation between $\Phi$ and $R$:
$$
V_i(\Phi^k_j)-V_j(\Phi^k_i)=3R^k_{ij}.
$$

Clearly $V(E)$  is integrable (as is $\Sp\{\Gamma\}$) while $H(E)$ is generally not.

\subsection{The Massa-Pagani connection}
Massa and Pagani \cite{MaPa94} introduced a linear connection $\hnabla$ on $E$ induced by $\Gamma$ by imposing some natural requirements. These include that the covariant differentials $\hat{\nabla}dt, \hat{\nabla}S$ and $\hat{\nabla}\Gamma$ are all zero, that the vertical sub-bundle is flat and conditions on some components of the torsion, $T$, and the curvature $R$. However, no explicit intrinsic formula for $\hnabla$ is given. In ~\cite{JP01} the connection is derived in a different manner and, while an intrinsic formula for $\hnabla_XY$ is used, it involves a subsidiary connection on a certain pullback bundle. We now give an intrinsic form of the Massa-Pagani connection using proposition \ref{propn1} and theorem \ref{thm1}.
Firstly, we need the (1,1) tensor $Q$ on $E$, taking values in $H(E)$ with $Q\circ S=P_H$ and $S\circ Q=P_V$, so that
$$Q=\psi^i\otimes H_i.$$
In passing we note that, on vector fields,
\begin{equation}\label{LGQ}
\mathcal{L}_\Gamma Q=[\Phi,Q] \quad \text{and}\quad \mathcal{L}_\Gamma Q\circ S=\Phi,
\end{equation}
where $[\Phi,Q]$ is the commutator of $\Phi$ and $Q$.

Quite remarkably, each of the eigendistributions of $\mathcal{L}_\Gamma S$ has a natural covariant derivative (there may be others). These will allow us to construct the covariant derivative for the Massa-Pagani connection using theorem \ref{thm1}.

\begin{lem}
The distributions $\Img(P_\Gamma), \Img(P_H)$ and $\Img(P_V)$ have the following covariant derivatives over $\mathfrak{F}(E)$
\begin{equation*}
\nabla^\Gamma_XY:=X(dt(Y))\Gamma,\quad \nabla^H_XY:=Q([X, S(Y)]),\quad \nabla^V_XY:=S([X,Q(Y)]).
\end{equation*}
\end{lem}

\begin{proof}
That $\nabla^\Gamma$ satisfies properties \eqref{cov deriv props} is self-evident. Now, for $X,Y \in H(E),$
\[
\nabla^H_{fX}Y=Q([fX,S(Y)])=fQ([X,S(Y)])-S(Y)(f)Q(X)=f\nabla^H_XY\ \ \text{since}\ \ Q(X)=0,
\]
and
\begin{align*}
\nabla^H_X(fY)&=Q([X,S(fY)])=X(f)Q(S(Y))+fQ([X,S(Y)])\\
&=X(f)Y+f\nabla^H_XY\ \ \text{since}\ \ Q(S(Y))=Y.
\end{align*}
The result for $\nabla^V$ follows similarly.
\end{proof}

Note that each of $\nabla^\Gamma,\nabla^H,\nabla^V$ maps into their respective distributions so that corollary \ref{cor2.3} applies.\newline
Denoting by $\bnabla^\Gamma, \bnabla^H$ and $\bnabla^V$ the extensions of $\nabla^\Gamma, \nabla^H$ and $\nabla^V$ given in proposition \ref{propn1}, we have

\begin{thm}\label{M-P conn1}
The Massa-Pagani connection is, for $X,Y\in\mathfrak{X}(E),$ 
\begin{equation}\label{MP formula1}
\hnabla_XY=\bnabla^\Gamma_X(P_{\Gamma}(Y))+\bnabla^H_X(P_H(Y))+\bnabla^V_X(P_V(Y)).
\end{equation}
Explicitly,
\begin{align}\label{MP formula2}
\notag\hnabla_XY= &\ X(dt(Y))\Gamma +Q([P_H(X),S(Y)])+S([P_V(X),Q(Y)])\\
&+P_H([X-P_H(X),P_H(Y)])+P_V([X-P_V(X),P_V(Y)]).
\end{align}
\end{thm}
Note that $\hnabla$ preserves our direct sum decomposition of $\mathfrak{X}(E).$
\begin{proof}
Theorem \ref{thm1} shows that \eqref{MP formula1} is indeed a covariant derivative and a little manipulation of the first term in \eqref{MP formula1} is required to produce the first term in \eqref{MP formula2}. The non-zero basis components are:
\begin{align}
\notag &\hnabla_\Gamma H_i = \Gamma^j_i H_j,     &  &\hnabla_\Gamma V_i  = \Gamma^j_i V_j,\\
\label{covariant-deriv-defn-basis} &\hnabla_{H_i}H_j = \frac{\partial \Gamma^k_i}{\partial u^j} H_k,  &  &\hnabla_{H_i}V_j = \frac{\partial \Gamma^k_i}{\partial u^j} V_k,
\end{align}
in agreement with those of the Massa-Pagani connection given in \cite{JP01,MaPa94}.
\end{proof}
The geometric properties of $\hnabla$ are captured in the following proposition whose conditions are largely those of Massa and Pagani.

\begin{propn}\label{M-P conn2}
The Massa-Pagani connection \eqref{MP formula1} is the unique linear connection on $E$  satisfying

\begin{align}\label{MP cdtns}
\notag &\hnabla \Gamma=0, &&\hnabla \d t=0,\\
\notag &\hnabla S=0,&&\hnabla Q=0,\\
&P_H(X)=\hat T(\Gamma,S(X)),&&P_V(X)=S(\hat T(\Gamma,X)),\\
\notag &R(\Gamma,V)=0,&&\hnabla_VU=0.
\end{align}

Here $V$ is vertical and $U$ is vertical and constant along the (vertical) fibres, that is, a vertical lift from $\R\times M.$
\end{propn}

\begin{proof}
The first four properties and the last are straightforward using \eqref{MP formula1},\eqref{MP formula2} and corollary \ref{cor1}.
The fifth property can be established as follows: because of the linearity of $\hat T$ and $S$ we only need to consider $\hat T(\Gamma,S(X))$ for $X\in\Img(P_H).$
Then, for $X\in\Img(P_H),$
\begin{align*}
\hat T(\Gamma,S(X))=&-P_H([\Gamma,S(X))])\ \text{from corollary \ref{cor1}}\\
=&-P_H(\mathcal{L}_\Gamma S(X)+S([\Gamma,X]))\\
=&\ P_H(X)
\end{align*}
recalling that $\Img(P_H)$ is the eigenspace of $\mathcal{L}_\Gamma S$ with eigenvalue $-1.$\newline
Property six follows in a similar fashion. The $R$ property is easily established by direct computation with the components \eqref{covariant-deriv-defn-basis} and the basis bracket relations.\newline
Thus we have established that there exists a linear connection satisfying the conditions \eqref{MP cdtns}. This also verifies that these conditions are self-consistent.
We now prove uniqueness by showing that \eqref{MP cdtns} implies \eqref{covariant-deriv-defn-basis}.

Firstly note that $\hnabla_X$ commutes with both $\PH$ and $\PV$ because $\hnabla S=0=\hnabla Q$ and $\PH=Q\circ S$ and $\PV=S\circ Q.$ Along with $\hnabla P_\Gamma =0$  this means that $\hnabla$ preserves our direct sum decomposition of  $TE.$

Now consider $\hnabla_\Gamma X$. The property $\hnabla S=0$ gives
\[
\hnabla_\Gamma V_i=\hnabla_\Gamma(S(H_i))=S(\hnabla_\Gamma H_i).
\]

Now we apply $\PH(X)=T(\Gamma,S(X))$ with $X=H_i$:
\begin{align*}
H_i&=\PH(H_i)=T(\Gamma,S(H_i))=\hnabla_\Gamma V_i - [\Gamma, V_i]\\
&\implies \hnabla_\Gamma V_i=\Gamma^j_i V_j\quad \text{and hence}\quad \hnabla_\Gamma H_i=\Gamma^j_i H_j.
\end{align*}

Next we show that $\hnabla_{V_i}H_j$=0. 
\[
S(\hnabla_{V_i}H_j)=\hnabla_{V_i}(S(H_j))=\hnabla_{V_i}V_j=0
\]
but $\hnabla_{V_i}H_j$ is horizontal and so it is zero. 
Now consider $\hnabla_{H_i}X$. Again using $\hnabla S=0$
\[
S(\hnabla_{H_i}H_j)=\hnabla_{H_i}(S(H_j))=\hnabla_{H_i}V_j.
\]

 The curvature condition $R(\Gamma,V_i)=0$ gives the values of these components as follows.
 \begin{align*}
 0=R(\Gamma,V_i)H_j&=\hnabla_\Gamma(\hnabla_{V_i}H_j)-\hnabla_{V_i}(\hnabla_\Gamma H_j)-\hnabla_{[\Gamma,V_i]}H_j\\
 &=-\hnabla_{V_i}(\Gamma^k_jH_k)-\hnabla_{-H_i+\Gamma^l_iV_l}H_j\\
 &=-\hnabla_{V_i}(\Gamma^k_jH_k)-\hnabla_{-H_i}H_j\\
 \implies\hnabla_{H_i}H_j&=V_i(\Gamma^k_j)H_k \quad\text{and so}\quad \hnabla_{H_i}V_j=V_i(\Gamma^k_j)V_k,
 \end{align*}
 which completes the proof.
\end{proof}

Note: The condition $P_H(X)=T(\Gamma,S(X))$ (from \cite{MaPa94}) is redundant in \eqref{MP cdtns}.

Recasting \eqref{MP formula2} we obtain the shape map \eqref{A-Gamma-defn} for the Massa-Pagani connection in an explicitly linear form:
\begin{align}\label{A_X formula}
A_X(Y)&=Y(X(t))\Gamma+(\PH\circ\lie{X}\PH)(Y)+(\PV\circ\lie{X}\PV)(Y)\\
              &-(\PH\circ\lie{\PH(X)}Q\circ S)(Y)-(\PV\circ\lie{\PV(X)}S\circ Q)(Y).
\end{align}
Most importantly in what follows, from corollary \ref{cor1} or from \eqref{A_X formula} by noting that $\lie{\Gamma}\PH=-\lie{\Gamma}\PV$,
\begin{equation}
A_\Gamma(X)=-\PH([\Gamma,\PV(X)])-\PV([\Gamma,\PH(X)]),
\end{equation}
that is,
\begin{equation}
A_\Gamma=-\PH\circ\lie{\Gamma}\PV - \PV\circ\lie{\Gamma}\PH
\end{equation}
and so, from \eqref{LGQ},
\begin{equation}\label{Shape Map(1)}
A_\Gamma = -\mathcal{L}_\Gamma Q\circ S+Q=-\Phi + Q= -\Phi^i_j\theta^j\otimes V_i+\psi^i\otimes H_i.
\end{equation}
The normal forms  of the component matrix, $\mathbf \Phi=(\Phi^i_j)$, of $\Phi$ are fundamental to the analysis of the inverse problem in the calculus of variations. While the (1,1) tensor  $\Phi$ itself clearly has no real eigenspaces, $A_\Gamma$ captures the real eigenspaces of  $\Phi^i_j$:
$$A_\Gamma(X)=\mu X \iff \mu^2\theta^i(X)=-\Phi^i_j\theta^j(X) \ \text{and}\ \psi^i(X)=\mu\theta^i(X).$$

The torsion, $\hat T$, of the Massa and Pagani connection contains all the important properties of the \SODE\ $\Gamma$:
\begin{align*}
&\hat T(\Gamma,V_i)=H_i,\ \hat T(\Gamma,H_i)=-\Phi^j_iV_j,\ \hat T(V_i,V_j)=0,\\ &\hat T(V_i,H_j)=0,\ \hat T(H_i,H_j)=-R^k_{ij}V_k.
\end{align*}
The vector field brackets are now more elegantly expressed as
\begin{subequations}\label{CD-VF-commutators}
\begin{align}
& [\Gamma, X^V] = \hnabla_\Gamma X^V - A_\Gamma(X^V),  \label{CD-VF-commutators-1}\\
& [\Gamma, X^H] =\hnabla_\Gamma X^H - A_\Gamma (X^H),   \label{CD-VF-commutators-2}\\
&  [X^V, Y^V] = \hnabla_{X^V} Y^V - \hnabla_{Y^V} X^V,    \label{CD-VF-commutators-3}\\
& [X^V, Y^H] = \hnabla_{X^V} Y^H - \hnabla_{Y^H} X^V,    \label{CD-VF-commutators-4}\\
& [X^H, Y^H] = \hnabla_{X^H} Y^H - \hnabla_{Y^H} X^H + R(X^H,Y^H),  \label{CD-VF-commutators-5}
\end{align}
\end{subequations}
where the superscripts $V$ and $H$ indicate elements of $V(E)$ and $H(E)$ respectively (and not vertical and horizontal lifts). Or we could have replaced $X^H$ and $X^V$ by $\PH(X)$ and $\PV(X)$.

\section{Constrained SODEs}

Now we turn to systems of second order equations with first order constraints on some of the dependent variables. Specifically, we consider
\begin{align}
\ddot x^a & = F^a (t, x^b,x^\beta, \dot x^b), \label{CSODE1}\\
\dot x^\alpha & =\Psi^\alpha(t,x^b,x^\beta,\dot x^b) \label{CSODE2}
\end{align}
where $a, b = 1, \dots, m$ and $\alpha,\beta = m+1,\dots,n.$ The $(x^\alpha)$

 can be thought of as coordinates on a constraint submanifold, as in the case of motion on a sphere, or as additional control variables. We show how these constraint equations arise in non-holonomic mechanics in subsection \ref{non-hol} followed by examples in section \ref{xmpls}.

The geometric setting is the (extended) configuration space $\R\times M$ with coordinates $(t,x^a,x^\alpha)$ and evolution space, $E$,  with adapted coordinates $(t,x^a,x^\alpha,u^a,u^\alpha)$. We will use the combined index $i=(a,\alpha)$ for brevity and compatibility with the unconstrained case. In order to preserve the dependence of \eqref{CSODE1} and \eqref{CSODE2} on the constraint variables under coordinate transformations and hence the tensorial character of the geometric constructs,  we give $M$ a product structure: $M:=M^m\times M^{n-m}.$ As we will see, $M^{n-m}$ also requires a linear connection. Wherever possible we use the notation of Sarlet {\em et al} \cite{SCS95,SCS97,SSC96}.

Again, we identify $E$ with $J^1(\R,M)$ and the contact system on $J^2(\R,M)$ is
$$
\Omega^2(\R,M)=\{dx^i-u^idt,du^i-v^idt\}.
$$
Now we identify equations \eqref{CSODE1},\eqref{CSODE2} with the embedded submanifold $\tilde{\mathcal F}$ of $J^2(\R,M)$ of co-dimension~$n$
$$
\tilde{\mathcal F}:=\{q\in J^2(\R,M):v^a(q)-F^a(q)=0, u^\alpha(q)-\Psi^\alpha(q)=0\}.
$$
The contact distribution restricted to $\tilde{\mathcal F}$ can be shown to be
$$
i_{\tilde{\mathcal F}}^\ast \Omega^2(\R,M)=\Sp\{\theta^a:=dx^a-u^adt, \theta^\alpha:=dx^\alpha-\Psi^\alpha dt,\phi^b:=du^b-F^bdt, \phi^\alpha:=d\Psi^\alpha-\dot\Psi^\alpha dt\}
$$
where  $i_{\tilde{\mathcal F}}$ is the inclusion map $\tilde{\mathcal F} \hookrightarrow J^2(\R,M)$ and
$$
\dot\Psi^\alpha:=\pd{\Psi^\alpha}{t}+u^a\pd{\Psi^\alpha}{x^a}
+\Psi^\beta\pd{\Psi^\alpha}{x^\beta}+F^b\pd{\Psi^\alpha}{u^b}
$$
is the total time derivative of $\Psi^\alpha.$
Now $\phi^\alpha\in \Sp\{\theta^j,\phi^b\}$ so the contact distribution restricted to $\tilde{\mathcal F}$  is generated by this submodule. The corresponding annihilated distribution is generated by the  differential equation field
\begin{equation}\label{CSODE}
\tGamma := \frac{\partial}{\partial t} + u^a
\frac{\partial}{\partial x^a} + \Psi^\beta\pd{}{x^\beta}+F^b \frac{\partial}{\partial u^b}.
\end{equation}
Unlike the unconstrained case  we cannot identify $i_{\tilde{\mathcal F}}$ as a section of $\pi^1_2:J^2(\R,M)\to J^1(\R,M).$ Instead it is a section of $\tilde\pi^1_2:\tilde{\mathcal F}\to \tilde E$ where $\tilde E$ is the embedded constraint submanifold
$$
\tilde E:=\{p\in E: u^\alpha(p)-\Psi^\alpha(p)=0\}.
$$
So we will now regard both the restricted contact distribution and $\tGamma$ as living on $\tilde E$ and we will use local coordinates $(t,x^i,u^a)$ on $\tilde E.$
Then there is the matter of the ``vertical" distribution on $\tilde E$. By virtue of the embedding of $\tilde E$ into $E$, at every point $p\in \tilde E$ we have  $\Sp\{\tilde V_a:=\pd{}{u^a}\} \subset T_p\tilde E$, and we denote the corresponding sub-bundle $\tilde V(\tilde E)$. 

Having both contact and vertical distributions we can once again construct the vertical endomorphism $S=\theta^a\otimes \tV_a$. It is straightforward to show that $S$ is tensorial under coordinate transformations which respect the product structure of $M.$  Computing $\lie{\tGamma}S$ we obtain eigenvalues $0,1$ and $-1$ as before but this time the eigendistributions are respectively, $\Sp\{\tGamma,\pd{}{x^\alpha}\}$, $\Sp\{\tV_a\}$ and $\Sp\{\tilde H_b\}$, where
$$
\tilde H_a:=\frac{\partial}{\partial x^a} -\tGamma^b_a\frac{\partial}{\partial u^b}-\Psi^\beta_a\pd{}{x^\beta}
$$
with
$$
\tGamma_a^b:= -{\frac{1}{2}}\frac{\partial F^b}{\partial u^a}\ \text{and}\ \Psi^\beta_a:=-\pd{\Psi^\beta}{u^a}.
$$
Together these eigendistributions give a direct sum decomposition of $T\tilde E$.
The dual basis is $\{dt, \eta^\alpha,\tpsi^a, \theta^b\}$ where 
$$
\eta^\alpha:=\theta^\alpha+\Psi^\alpha_b\theta^b,\ \ \tpsi^a:=\phi^a + \tGamma^a_b\theta^b,\  \ \theta^b:= dx^b-u^bdt.
$$
We will use the following projectors:
$$P_{\tGamma}=dt\otimes\tGamma,\ \  P_N=N:=\eta^\alpha\otimes\pd{}{x^\alpha},\ \  P_V=\tpsi^a\otimes \tV_a,\  \ P_H=\theta^a\otimes\tilde H_a
$$
along with the $(1,1)$ tensors $Q:=\tpsi^a\otimes\tilde H_a$ and $S=\theta^a\otimes \tV_a.$

The non-zero brackets of the basis fields provide some interesting additions to the unconstrained case:
\begin{align}
\notag [\tGamma,\tilde H_a]&=\tilde \Phi^b_a\tV_b+\tGamma^b_a\tilde H_b+K^\alpha_a\pd{}{x^\alpha},
&&[\tGamma, \tV_a]=-\tilde H_a+\tGamma^b_a\tV_b,\\
[\tilde H_a,\tV_b]&=\pd{\tGamma^c_a}{u^b}\tV_c+\pd{\Psi^\alpha_a}{u^b}\pd{}{x^\alpha} = [\tilde H_b,\tV_a],
&&[\tilde H_a,\tilde H_b]=\hat R^c_{ab}\tV_c+\check R^\beta_{ab}\pd{}{x^\beta},\label{basis brackets}\\
\notag[\tGamma,\pd{}{x^\alpha}]&=-\pd{\Psi^\beta}{x^\alpha}\pd{}{x^\beta}-\pd{F^c}{x^\alpha}\tV_c,
&&[\tilde H_a,\pd{}{x^\alpha}]=\pd{\tGamma^b_a}{x^\alpha}\tV_b+\pd{\Psi^\beta_a}{x^\alpha}\pd{}{x^\beta}
\end{align}
where
\begin{align*}
\tilde\Phi^b_a&:=-\pd{F^b}{x^a}-\tGamma(\tGamma^b_a)-\tGamma^c_a\tGamma^b_c+\Psi^\alpha_a\pd{F^b}{x^\alpha},\\
K^\alpha_a&:=-\tGamma(\Psi^\alpha_a)-\tilde H_a(\Psi^\alpha)+\tGamma^b_a\Psi^\alpha_b,\\
R^d_{ab}&:=\frac{1}{2}\left (\frac{\partial^2 F^d}{\partial x^a\partial u^b} - \frac{\partial^2 F^d}{\partial x^b\partial u^a}+\frac{1}{2}\left (\frac{\partial F^c}{\partial u^a}\frac{\partial^2 F^d}{\partial u^c\partial u^b} -
\frac{\partial F^c}{\partial u^b}\frac{\partial^2 F^d}{\partial u^c\partial u^a} \right )\right ),\\
\hat R^c_{ab}&:=R^c_{ab}- \Psi^\beta_b \pd{\tGamma^c_a}{x^\beta}+\Psi^\beta_a\pd{\tGamma^c_b}{x^\beta},\\
\check R^\beta_{ab}&:=\left (\pd{\Psi^\beta_a}{x^b}-\pd{\Psi^\beta_b}{x^a}\right )+\left( \tGamma^c_a\pd{\Psi^\beta_b}{u^c}-\tGamma^c_b\pd{\Psi^\beta_a}{u^c} \right) +\left( \Psi^\alpha_a \pd{\Psi^\beta_b}{x^\alpha} - \Psi^\alpha_b \pd{\Psi^\beta_a}{x^\alpha} \right).
\end{align*}
Also,
\begin{align*}
\tV_a(\tilde \Phi^c_b)-\tV_b(\tilde\Phi^c_a)&=3\hat R^c_{ab}\\
\tV_a(K^\alpha_b)-\tV_b(K^\alpha_a)&=2\check R^\alpha_{ab}.
\end{align*}

\subsection{A geometric description of the constraints}

The discussion so far has assumed that constraints described by equations~\eqref{CSODE2} are given as a submanifold $\tilde{E}$ of the jet manifold $J^1(\R,M)$, where $M$ has a suitable product structure, but it is also possible to describe these constraints using a nonlinear connection on a related bundle. This second connection (which is independent of any \SODE) is based on a construction given in~\cite{SCS97} using maps between jet bundles. We shall describe it using a commutative diagram, and to save space we shall adopt the following notation. We shall let $\mu : \R \times M^m \to \R$ and $\rho : \R \times M \to \R \times M^m$ be the two projections, and write $J^1\mu$ and $J^1\rho$ for the spaces of jets of local sections of these bundles. (The projection $\mu$ is, of course, the same as the coordinate function $t$.) We shall also need two fibre product spaces:
\begin{align*}
C & = (\R \times M) \times_{\R \times M^m} J^1\mu = M^{n-m} \times J^1\mu, \\
B & = J^1\rho \times_{\R \times M^m} J^1\mu.
\end{align*}
In addition, writing $\alpha : C \to J^1\mu$ for the projection on the second component, we shall need the jet space $J^1\alpha$.
\begin{center}
\begin{tikzcd}
J^1\alpha \arrow{drr}{\alpha^0_1} \\
B \arrow{rr} \arrow{dd} \arrow[dashrightarrow]{u} \arrow[dashrightarrow]{dr} && C \arrow{r}{\alpha} \arrow{dd} & J^1\mu \arrow{dd}{\mu^0_1} \\
& J^1(\R, M) \arrow{dr} \arrow[dashrightarrow]{ur} \\
J^1\rho \arrow[swap]{rr}{\rho^0_1} && \R \times M \arrow[swap]{r}{\rho} & \R \times M^m \arrow[swap]{r}{\mu} & \R
\end{tikzcd}
\end{center}

The three dashed arrows in this diagram are constructed as follows.
\begin{itemize}
\item The map $B \to J^1\alpha$ is an injection, obtained by taking a general element $\bigl( j^1_{\gamma(t)} \phi, j^1_t \gamma \bigr) \in B$, where $\gamma$ is a local section of $\mu$ in a neighbourhood of $t$ (a curve in $M^m$) and $\phi$ is a local section of $\rho$ in a neighbourhood of $\gamma(t)$. We map this general element to $j^1_{j^1_t\gamma}(\phi \circ \mu^0_1, \id_{J^1\mu}) \in J^1 \alpha$, and we regard $B$ as a submanfold of $J^1\alpha$.
\item The map $B \to J^1(\R,M)$ is a projection, obtained by taking the general element of $B$ and mapping it to $j^1_t(\phi \circ \gamma)$.
\item The map $J^1(\R,M) \to C$ is also a projection, obtained by taking a general element $j^1_t\delta \in J^1(\R,M)$ and mapping it to $(\delta(t), j^1_t(\rho \circ \delta))$.
\end{itemize}
With this diagram in place, we can now see how a nonlinear connection on the bundle $\alpha : C \to J^1\mu$ determines a constraint submanifold $\tilde{E} \subset J^1(\R,M)$, and vice versa.
\begin{center}
\begin{tikzcd}
B \arrow[shift left]{rr} \arrow{ddr} && C \arrow[shift left]{ll}{\sigma} \arrow[shift right, swap]{ddl}{\Psi} \\ \\
& J^1(\R,M) \arrow[shift right]{uur}
\end{tikzcd}
\end{center}
A connection on a bundle may be specified by a section of the corresponding jet bundle. So let $\sigma : C \to B \subset J^1\alpha$ be such a section, chosen to take its values in the submanifold $B$. Composing this section with the projection $B \to J^1(\R,M)$ then gives a map $\Psi$ which is a section of the second projection $J^1(\R,M) \to C$, and so is a diffeomorphism onto its image $\tilde{E} = \Psi(C)$.

Conversely, starting with a constraint submanifold $\tilde{E} \subset J^1(\R,M)$ transverse to the projection $J^1(\R,M) \to C$, and hence with a section $\Psi$ whose image is $\tilde{E}$, we may construct the `Chetaev bundle' in the following way. Let $\overline{S}$ denote the vertical endomorphism on $E = J^1(\R,M)$. For any $p \in \tilde{E}$ consider the subspace $T_p \tilde{E} \subset T_p E$ and its annihilator $T_p^\circ \tilde{E} \subset T^*_p E$. Acting on each annihilator with $\overline{S}$ gives a co-distribution on $\tilde{E}$, and therefore on its diffeomorphic image $\Psi^{-1}(\tilde{E}) = C$, defining a connection on $\alpha : C \to J^1\mu$. We also see, using the diffeomorphism $\Psi$, that the vertical endomorphism on $J^1\mu$ gives rise naturally to an endomorphism $S$ on $\tilde{E} \cong C = M^{n-m} \times J^1\mu$.

We may see the effect of these constructions in coordinates. As before, we have coordinates $(t, x^i)$ on $M$, so that coordinates on
$J^1\mu$ are $(t, x^a, u^a)$ and those on $C$ are $(t, x^i, u^a)$. The jet coordinates on $J^1\alpha$ are the derivatives of $x^\alpha$ with respect to $(t, x^a, u^a)$, which we write as $(v^\alpha, v^\alpha_b, w^\alpha_b)$, and the submanifold $B \subset J^1\alpha$ is given by $w^\alpha_b = 0$. The submanifold coordinates on $B$ are not, though, adapted to the projection $B \to E$; we see instead from
$$
\fpd{(\phi^\alpha \circ \gamma)}{t} = \fpd{\phi^\alpha}{t} + \fpd{\phi^\alpha}{x^b} \frac{d\gamma^b}{dt}
$$
that the projection is given by $u^\alpha = v^\alpha + v^\alpha_b u^b$. On the other hand, though, the projection $E \to C$ simply discards the coordinates $u^\alpha$. If, therefore, we are given the connection $\sigma$ in terms of functions $(\sigma^\alpha, \sigma^\alpha_b)$, the map $\Psi : C \to E$ is then given by $\Psi^\alpha = \sigma^\alpha + \sigma^\alpha_b u^b$. This construction gives affine constraints, with $\Psi^\alpha$ as affine functions of $u^b$, precisely when the functions $\sigma^\alpha$, $\sigma^\alpha_b$ are independent of $u^b$, and this arises when the connection $\sigma$ on $\alpha$ is projectable to a connection on $\rho$.

Now suppose instead that we are given a constraint submanifold $\tilde{E} \subset E$ in terms of equations $u^\alpha = \Psi^\alpha(t, x^i, u^a)$. The annihilator bundle $T^\circ \tilde{E} \subset T^*_{\tilde{E}}E$ is spanned by the forms $du^\alpha - d\Psi^\alpha$ (defined only on $\tilde{E}$, of course), so that $\overline{S}(T^\circ \tilde{E})$ is spanned by the forms
$$
\theta^\alpha - \fpd{\Psi^\alpha}{u^b} \theta^b = \theta^\alpha + \Psi^\alpha_b \theta^b = \eta^\alpha.
$$
These define horizontal vector fields
\begin{equation}
\label{AuxConn}
\fpd{}{t} + (\Psi^\alpha + u^b \Psi^\alpha_b) \fpd{}{x^\alpha} , \qquad \fpd{}{x^b} - \Psi^\alpha_b \fpd{}{x^\alpha}, \qquad \fpd{}{u^a}
\end{equation}
on $C \cong \tilde{E}$, and hence a connection $\sigma$ taking values in $B$ with coordinate representation
$$
\sigma^\alpha = \Psi^\alpha + u^b \Psi^\alpha_b, \qquad \sigma^\alpha_b = - \Psi^\alpha_b.
$$

In the next Section we shall construct a linear connection on $\tilde{E}$ corresponding to a constrained \SODE. Part of this linear connection will be derived from a linearisation of the connection $\sigma$ obtained from the constraints $\tilde{E}$, together with an auxiliary linear connection $D$ on the manifold $M^{n-m}$. The result is a connection $\bnabla^N$ on the vertical bundle of $\alpha$, given by
$$
\bnabla^N_X Y = D_{N(X)} Y +N [ X - N(X), Y]
$$
where $X$ and $Y$ are vector fields on $C$ with $Y$ vertical over $J^1 \mu$, and where $N$ is the vertical projector of the connection $\sigma$.

\subsection{A linear connection for constrained {\sc \bf SODEs}}

The generalisation of the Massa--Pagani connection to the constrained case is achieved by equipping the base constraint manifold $M^{n-m}$ with a linear connection, not necessarily symmetric, as indicated above. However, we will not attempt to prove that these are the only generalisations. So, suppose that this linear connection has coefficients $\Upsilon^\gamma_{\alpha\beta}$ in the basis $\{\pd{}{x^\alpha}\}$ for $\X(M^{n-m}).$ This induces the covariant derivative $\nabla^N$ on $\Img(N):$
\[
\nabla^N_XY=(X(Y^\gamma)+\Upsilon^\gamma_{\alpha\beta}X^\alpha Y^\beta)\pd{}{x^\gamma}.
\]
As in the unconstrained case the other submodule covariant derivatives are (without proof)
\begin{equation*}
\nabla^\tGamma_XY:=X(dt(Y))\tGamma,\quad \nabla^{\tH}_XY:=Q([X, S(Y)]),\quad \nabla^{\tV}_XY:=S([X,Q(Y)]).
\end{equation*}
Using the constructions of proposition \ref{propn1} and theorem \ref{thm1} we have (using $\tnabla$ to distinguish this from $\hnabla$ of the unconstrained case)

\begin{thm}\label{Main result}
\begin{equation}\label{constrained conn.}
\tnabla_XY:=\bnabla^\tGamma_X(P_{\tGamma}(Y))+\bnabla^N_X(N(Y))+\bnabla^{\tH}_X(P_H(Y))+\bnabla^{\tV}_X(P_V(Y)),
\end{equation}
explicitly,
\begin{align}
\tnabla_XY= &\ X(dt(Y))\tGamma +\nabla^N_{N(X)}(N(Y))+Q([P_H(X),S(Y)])+S([P_V(X),Q(Y)])\\
\notag&+N([X-N(X),N(Y)])+P_H([X-P_H(X),P_H(Y)])+P_V([X-P_V(X),P_V(Y)])
\end{align}

is the unique linear connection on $\tilde E$ satisfying 
\begin{align}
\notag&\tnabla\tGamma=0,\ \tnabla dt=0,\ \tnabla S=0,\ \tnabla Q=0,\ \tnabla_{\tV_a}\tV_b=0,\ \\
&P_H(X)=T(\tGamma,S(X)),\ P_V(X)=S(T(\tGamma,X))\label{conn props}\\
\notag&R(\tGamma,\tV_a)=0,\ \tnabla N=0,\ \tnabla_{\pd{}{x^\alpha}}\pd{}{x^\beta}=\Upsilon^\gamma_{\alpha\beta}\pd{}{x^\gamma},\\
\notag&(I-P_V)(T((I-N)(X),N(Y)))=0.
\end{align}

\end{thm}
\begin{proof}

Equation \eqref{constrained conn.} describes a linear connection on $\tE$ by virtue of theorem \ref{thm1}.

The non-zero basis components calculated from \eqref{constrained conn.} and using \eqref{basis brackets} are:
\begin{align}
\notag \tnabla_{\tGamma} \tH_a &= \tGamma^b_a \tH_b,     &&\tnabla_{\tGamma} \tV_a = \tGamma^b_a \tV_b,\\
\label{conn cmpts} \tnabla_{\tH_a}\tH_b& = \frac{\partial \tGamma^c_{a}}{\partial u^b} \tH_c,  &&\tnabla_{\tH_a}\tV_b = \frac{\partial \tGamma^c_{a}}{\partial u^b} \tV_c, \\
\notag \tnabla_{\tGamma}\pd{}{x^\alpha}&=-\pd{\Psi^\beta}{x^\alpha}\pd{}{x^\beta}, && \tnabla_{\tH_a}\pd{}{x^\alpha}=\pd{\Psi^\beta_a}{x^\alpha}\pd{}{x^\beta},\\
\notag \tnabla_{\pd{}{x^\alpha}}\pd{}{x^\beta}&=\Upsilon^\gamma_{\alpha\beta}\pd{}{x^\gamma}&& \
\end{align}
(resembling the unconstrained case).

The second part of the proof entails showing that $\tnabla$ in \eqref{constrained conn.} satisfies the given properties \eqref{conn props}, which simultaneously establishes that at least one connection satisfies them and that they are self-consistent. Finally, we must show that the properties \eqref{conn props} produce the same components as \eqref{constrained conn.}, so establishing uniqueness.

The derivation of all but the last three properties are achieved from \eqref{constrained conn.} exactly as for proposition \ref{M-P conn2}.
Of the remaining properties $\tnabla N=0$ follows from corollary \ref{cor2.3}; the second last follows from the definition of $\nabla^N$ and finally, $(I-P_V)(T((I-N)(X),N(Y)))=0$ is established by observing from \eqref{basis brackets} and \eqref{conn cmpts} that
\begin{equation*}
T(\tGamma,\pd{}{x^\alpha})=\pd{F^c}{x^\alpha}\tV_c,\quad
T(\tV_a,\pd{}{x^\alpha})=0,\quad
T(\tH_a,\pd{}{x^\alpha})=-\pd{\tGamma^b_a}{x^\alpha}\tV_b,
\end{equation*}
or by an appeal to the first case of corollary \ref{cor1}.

To finish the proof we now show that properties \eqref{conn props} produce the components \eqref{conn cmpts}. In part this closely follows the proof of proposition \ref{M-P conn2}.

Firstly note that $\tnabla_X$ commutes with both $\PH$ and $\PV$ because $\tnabla S=0=\tnabla Q$ and $\PH=Q\circ S$ and $\PV=S\circ Q.$ Along with $\tnabla P_\tGamma =0$ and $\tnabla N=0$ this means that $\tnabla$ preserves our direct sum decomposition of  $T\tilde E.$

Now consider $\tnabla_\tGamma X$. The property $\tnabla S=0$ gives
\[
\tnabla_\tGamma \tV_a=\tnabla_\tGamma(S(\tH_a))=S(\tnabla_\tGamma\tH_a)
\]

Now we apply $\PH(X)=T(\tGamma,S(X))$ with $X=\tH_a$:
\begin{align*}
\tH_a&=\PH(\tH_a)=T(\tGamma,S(\tH_a))=\tnabla_\tGamma \tV_a - [\tGamma, \tV_a]\\
&\implies \tnabla_\tGamma \tV_a=\tGamma^b_a \tV_b\quad \text{and hence}\quad \tnabla_\tGamma \tH_a=\tGamma^b_a \tH_b.
\end{align*}
Now consider $(I-P_V)(T((I-N)(X),N(Y))=0$ with $X=\tGamma$ and $Y=\pd{}{x^\alpha},$ remembering that $\tnabla_X\pd{}{x^\alpha}\in \text{Im}(N)$ and using \eqref{basis brackets}:
\begin{align*}
&(I-P_V)(T(\tGamma,\pd{}{x^\alpha}))=0\\
\implies &(I-P_V)(\tnabla_\tGamma\pd{}{x^\alpha}-[\tGamma,\pd{}{x^\alpha}])=0\\
\implies&\tnabla_\tGamma\pd{}{x^\alpha}=N([\tGamma,\pd{}{x^\alpha}])\\
\implies &\tnabla_\tGamma\pd{}{x^\alpha}=-\pd{\Psi^\beta}{x^\alpha}\pd{}{x^\beta}.
\end{align*}

Next we show that both $\tnabla_{\tV_a}\tH_b$ and $\tnabla_{\tV_a}\pd{}{x^\alpha}$ are zero.
\[
S(\tnabla_{\tV_a}\tH_b)=\tnabla_{\tV_a}(S(\tH_b))=\tnabla_{\tV_a}\tV_b=0\ \text{by assumption}
\]
but $\tnabla_{\tV_a}H_b$ is horizontal and so it is zero. And
\[
0=(I-P_V)(T(\tV_a,\pd{}{x^\alpha}))=(I-P_V)(\tnabla_{\tV_a}\pd{}{x^\alpha}-\tnabla_{\pd{}{x^\alpha}}\tV_a)=\tnabla_{\tV_a}\pd{}{x^\alpha}.
\]
Now consider $\tnabla_{\tH_a}X$. Again using $\tnabla S=0$
\[
S(\tnabla_{\tH_a}\tH_b)=\tnabla_{\tH_a}(S(\tH_b))=\tnabla_{\tH_a}\tV_b.
\]

 The curvature condition $R(\tGamma,\tV_a)=0$ gives the values of these components as follows.
 \begin{align*}
 0=R(\tGamma,\tV_a)\tH_b&=\tnabla_\tGamma(\tnabla_{\tV_a}\tH_b)-\tnabla_{\tV_a}(\tnabla_\tGamma \tH_b)-\tnabla_{[\tGamma,\tV_a]}\tH_b\\
 &=-\tnabla_{\tV_a}(\tGamma^c_b\tH_c)-\tnabla_{-\tH_a+\tGamma^d_a\tV_d}\tH_b\\
 &=-\tnabla_{\tV_a}(\tGamma^c_b\tH_c)-\tnabla_{-\tH_a}\tH_b\\
 \implies\tnabla_{\tH_a}\tH_b&=\tV_a(\tGamma^c_b)\tH_c \quad\text{and so}\quad \tnabla_{\tH_a}\tV_b=\tV_a(\tGamma^c_b)\tV_c.
 \end{align*}
 To confirm $\tnabla_{\tH_a}\pd{}{x^\alpha}=\pd{\Psi^\beta_a}{x^\alpha}\pd{}{x^\beta},\ \tnabla_{\pd{}{x^\alpha}}\tH_a=0$
 and $\tnabla_{\pd{}{x^\alpha}}\tV_a=0$ we use
 \[
 0=(I-P_V)(T(\tH_a,\pd{}{x^\alpha}))=(I-P_V)(\tnabla_{\tH_a}\pd{}{x^\alpha}-\tnabla_{\pd{}{x^\alpha}}\tH_a-[\tH_a,\pd{}{x^\alpha}])
 \]
 Hence
 \[
 \tnabla_{\tH_a}\pd{}{x^\alpha}=N([\tH_a,\pd{}{x^\alpha}])=\pd{\Psi^\beta_a}{x^\alpha}\pd{}{x^\beta}\quad\text{and}\quad \tnabla_{\pd{}{x^\alpha}}\tH_a=P_H([\tH_a,\pd{}{x^\alpha}])=0.
 \]
Finally,
\[
\tnabla_{\pd{}{x^\alpha}}\tV_a=\tnabla_{\pd{}{x^\alpha}}(S(\tH_a))=S(\tnabla_{\pd{}{x^\alpha}}\tH_a)=0
\]

which completes the proof.
\end{proof}

Note: The condition $P_H(X)=T(\tGamma,S(X))$ is again redundant in the constrained case.

The shape map is (compare with \eqref{Shape Map(1)} in the unconstrained case)
\[
A_\tGamma=-\mathcal{L}_\tGamma N\circ N-\mathcal{L}_\tGamma Q\circ S+Q
\]
where, in this case,
\[
\mathcal{L}_\tGamma Q\circ S=\tilde \Phi + K
\]
with $\tilde\Phi=\tilde\Phi^a_b\theta^b\otimes \tV_a$ and $K=K^\alpha_a\theta^a\otimes \pd{}{x^\alpha}.$
As a result
\[
A_\tGamma(\tGamma)=0,\quad A_\tGamma(\tV_a)=\tH_a,\quad A_\tGamma(\tH_a)=-\tilde \Phi^b_a\tV_b-K^\alpha_a\pd{}{x^\alpha},\quad A_\tGamma(\pd{}{x^\alpha})=\pd{F^c}{x^\alpha}\tV_c.
\]
So, if $X=X^\alpha\pd{}{x^\alpha}+\bar X^a\tH_a+\hat X^a\tV_a$ is an eigenvector of $A_\tGamma$ belonging to $\mu$ then
\begin{align*}
\mu X^\alpha&=-\bar X^aK^\alpha_a,\\
\mu \bar X^a&=\hat X^a,\\
\mu\hat X^b&=X^\alpha\pd{F^b}{x^\alpha}-\tilde\Phi^b_a\bar X^a.
\end{align*}
A little manipulation shows that the necessary condition for non-zero solutions $X$ is
\begin{equation}
\det(\mu^3I+\Lambda_\mu)=0\quad
\text{where}\quad (\Lambda_\mu)^a_b:=K^\alpha_b\pd{F^a}{x^\alpha}+\mu\Phi^a_b.
\end{equation}

The role of the normal forms of $\Lambda_\mu$ in the classification of constrained systems is yet to be explored.

One important case is revealed by considering the geometry of the shape maps associated with the constraint submanifold,

\[
A_{\pd{}{x^\alpha}}(X)=\tnabla_{\pd{}{x^\alpha}}X-\left[\pd{}{x^\alpha},X\right].
\]
Notice that
\[
A_{\pd{}{x^\alpha}}(\tGamma)=[\tGamma,\pd{}{x^\alpha}]=-\pd{\Psi^\beta}{x^\alpha}\pd{}{x^\beta}-\pd{F^c}{x^\alpha}V_c
\]
hence
\begin{equation}\label{constraint deformation}
A_{\pd{}{x^\alpha}}(\tGamma)=0 \iff \pd{\Psi^\beta}{x^\alpha}=0=\pd{F^c}{x^\alpha}.
\end{equation}
In addition,
\[
A_{\pd{}{x^\alpha}}(\tGamma)=0 \implies A_{\pd{}{x^\alpha}}(\tH_a)=0
\]
while $A_{\pd{}{x^\alpha}}(\pd{}{x^\beta})=\Upsilon^\gamma_{\alpha\beta}\pd{}{x^\gamma}.$
The equivalence \eqref{constraint deformation} shows how the absence of $x^\alpha$ in equations \eqref{CSODE1}, \eqref{CSODE2} decouples the tangent space deformations induced by the constraints from the dynamics.

As in the unconstrained case, the torsion contains all the features of interest:
\begin{align*}
&T(\tGamma, \tV_a)=\tH_a, &&T(\tGamma, \tH_a)=-\tPhi^b_a\tV_b-K^\alpha_a\pd{}{x^\alpha}, &&&T(\tGamma,\pd{}{x^\alpha})=\pd{F^c}{x^\alpha}\tV_c, \\
&T(\tV_a,\tH_b)=\pd{\Psi^\alpha_b}{u^a}\pd{}{x^\alpha},
&&T(\tH_a,\tH_b)=-{\hat R}^c_{ab}\tV_c - {\check R}^\beta_{ab}\pd{}{x^\beta},
&&&T(\tH_a,\pd{}{x^\alpha})=-\pd{\tGamma^b_a}{x^\alpha}\tV_b,\\ &\ &&T(\pd{}{x^\alpha},\pd{}{x^\beta})=(\Upsilon^\gamma_{\alpha\beta}-\Upsilon^\gamma_{\beta\alpha})\pd{}{x^\gamma},
\end{align*}
all other components being zero.

\subsection{Nonholonomic dynamics  and the Lagrange-d'Alembert Principle}\label{non-hol}\ 
As discussed in the introduction one of the most interesting examples of application of our theory for constrained SODEs comes from nonholonomic mechanics \cite{NeiFuf72,Arnold,Bloch03,Cortes2002}. See also \cite{IbMa91} for other applications.

Assume that the system is subjected to  a set of non-holonomic constraints determined by a submanifold ${\tilde E}$ of $J^1({\mathbb R}, M)$ of codimension  $m$ (typically in nonholonomic dynamics ${\tilde E}$ will be an affine subbundle of $J^1({\mathbb R}, M)$).

This means that the only allowable evolutions are sections $\phi$ of $t: {\mathbb R}\times M\rightarrow {\mathbb R}$, $j_t^1\phi$  such that $j_t^1\phi\in {\tilde E}$ for all $t$.

If we take coordinates $(t, x^a, x^{\alpha})$ on ${\mathbb R}\times M$ and induced coordinates $(t, x^a, x^{\alpha}, \dot{x}^a, \dot{x}^\alpha)$ on $J^1({\mathbb R}, M)$ then we can  assume that locally $\tilde{E}$ is defined by the vanishing of the constraints (\ref{CSODE2})
\[
\dot x^\alpha  =\Psi^\alpha(t,x^b,x^\beta,\dot x^b).
\]
(We use $\dot x^i$ instead of $u^i$ in the context of nonholonomic mechanics.)

Consider a Lagrangian function $L: J^1({\mathbb R}, M)\rightarrow  {\mathbb R}$ with positive definite Hessian, then the equations of motion of a nonholonomic system are determined by a generalisation of the classical Lagrange-d'Alembert principle.
\begin{defn}(Chetaev principle \cite{Chetaev1934})
	The  solutions of a nonholonomic system $(L, \tilde{E})$  are sections $\phi$ of $t: {\mathbb R}\times M \rightarrow {\mathbb R}$ such that $j_t^1\phi\in {\tilde E}$  and satisfy
	\[
	\delta \int_{0}^T L(t, x(t), \dot{x}(t))\; dt =0
	\]
	where the variations are constrained by:
	$\displaystyle{
	\delta x^{\alpha}=\frac{\partial \Psi^\alpha}{\partial \dot{x}^a}\delta x^a}$.
	
\end{defn}

Using the usual arguments of variational calculus  we obtain the nonholonomic equations:
\begin{equation*}\label{nh-eq}
\begin{split}
\frac{d}{dt}\left( \frac{\partial L}{\partial \dot{x}^a}\right)-\frac{\partial L}{\partial x^a} =&-\lambda_\alpha \frac{\partial \Psi^\alpha}{\partial \dot{x}^a},\\
\frac{d}{dt}\left( \frac{\partial L}{\partial \dot{x}^\alpha}\right)-\frac{\partial L}{\partial x^\alpha} =&\lambda_\alpha, \\
\dot x^\alpha  -\Psi^\alpha(t,x^b,x^\beta,\dot x^b)=&0
\end{split}
\end{equation*}
where $\lambda_{\alpha}$ are Lagrange multipliers to be determined.


These equations are equivalent to
\begin{equation}\label{nh-eq-1}
\begin{split}
&\frac{d}{dt}\left( \frac{\partial L}{\partial \dot{x}^a}\right)-\frac{\partial L}{\partial x^a} +
\left[\frac{d}{dt}\left( \frac{\partial L}{\partial \dot{x}^\alpha}\right)-\frac{\partial L}{\partial x^\alpha}\right]\frac{\partial \Psi^\alpha}{\partial \dot{x}^a}=0 \; , \quad 1\leq a\leq m, \\
&\dot x^\alpha  -\Psi^\alpha(t,x^b,x^\beta,\dot x^b)=0, \qquad m+1\leq \alpha\leq n.
\end{split}
\end{equation}

Now, by first developing  the first equation of  (\ref{nh-eq-1}) and using the time  derivative of the second equation along solutions we obtain
\begin{equation}\label{nh-eq-2}
\begin{split}
&\left[\frac{\partial^2 L}{\partial \dot{x}^a\partial \dot{x}^b}
+\frac{\partial^2 L}{\partial \dot{x}^a\partial \dot{x}^\alpha}\frac{\partial \Psi^{\alpha}}{\partial \dot{x}^b}
+\frac{\partial^2 L}{\partial \dot{x}^b\partial \dot{x}^\alpha}\frac{\partial \Psi^{\alpha}}{\partial \dot{x}^a}
+\frac{\partial^2 L}{\partial \dot{x}^\alpha\partial \dot{x}^\beta}\frac{\partial \Psi^{\alpha}}{\partial \dot{x}^a}\frac{\partial \Psi^{\beta}}{\partial \dot{x}^b}\right]\ddot{x}^b+ G_a(t,x^b,x^\beta,\dot x^b)=0\; .\\
\end{split}
\end{equation}
Denote by
\[
{\mathcal C}_{ab}=\left[\frac{\partial^2 L}{\partial \dot{x}^a\partial \dot{x}^b}
+\frac{\partial^2 L}{\partial \dot{x}^a\partial \dot{x}^\alpha}\frac{\partial \Psi^{\alpha}}{\partial \dot{x}^b}
+\frac{\partial^2 L}{\partial \dot{x}^b\partial \dot{x}^\alpha}\frac{\partial \Psi^{\alpha}}{\partial \dot{x}^a}
+\frac{\partial^2 L}{\partial \dot{x}^\alpha\partial \dot{x}^\beta}\frac{\partial \Psi^{\alpha}}{\partial \dot{x}^a}\frac{\partial \Psi^{\beta}}{\partial \dot{x}^b}\right] \; .
\]

If we assume that the Hessian matrix
\[
(W_{ij})=\left( \frac{\partial^2 L}{\partial \dot{x}^i\partial\dot{x}^j}\right)
\]
 is positive definite, then pointwise
the matrix $
({\mathcal C}_{ab})
$, $1\leq a, b\leq m$,
represents the    restriction of a  positive definite quadratic form  given by the matrix $(W_{ij})$ to the subspace spanned by the $m$ linear independent vectors
\[
\begin{split}
&(1,0,\ldots, 0;  \frac{\partial \Psi^{m+1}}{\partial \dot{x}^1}, \ldots, \frac{\partial \Psi^{n}}{\partial \dot{x}^1}) \; , \\
&(0,1,\ldots, 0;  \frac{\partial \Psi^{m+1}}{\partial \dot{x}^2}, \ldots, \frac{\partial \Psi^n}{\partial \dot{x}^2}) \; ,\dots, \\
&(0,0,\ldots, 1;  \frac{\partial \Psi^{m+1}}{\partial \dot{x}^m}, \ldots, \frac{\partial \Psi^{n}}{\partial \dot{x}^m}). \;
\end{split}
\]
Therefore, from the  positive definiteness of $(W_{ij})$,  we deduce the positive definiteness of the restriction to any subspace and, in particular,  we deduce the regularity of the matrix $({\mathcal C}_{ab})$.

%
 Then  we can write (\ref{nh-eq-1}) as equations (\ref{CSODE1}) and  (\ref{CSODE2}):
\begin{eqnarray*}
	\ddot{x}^a&=&F^a(t, x^b, x^{\beta}, \dot{x}^b),\\
	\dot{x}^{\alpha}&=&\Psi^{\alpha}(t, x^b, x^{\beta}, \dot{x}^b).
\end{eqnarray*}


\section{Examples}\label{xmpls}

There are many examples in the literature of the Massa-Pagani connection for the unconstrained case (in both fixed and arbitrary dimension) so we restrict our attention to our new connection for the constrained scenario. Since the examples come from nonholonomic mechanics we will again use $\dot x^i$ instead of $u^i.$

\begin{ex}[Knife edge]
	Consider a sled on the plane restricted to move in the direction of its orientation. The configuration manifold is $(S^1\times\R)\times\R$ with coordinates $(\phi,x,y)$, where $(x,y)$ denotes the position of the contact point on the plane and $\phi$ denotes the orientation. The system has Lagrangian $L=\frac{1}{2}(\dot{x}^2+\dot{y}^2+\dot{\phi}^2)$ and nonholonomic constraint $\dot{y}=\tan(\phi)\dot{x}$, see \cite{Bloch03}. The nonholonomic equations are
	\begin{eqnarray*}
		\ddot{\phi} &=&0 \, , \\
		\ddot{x} &=& -\dot{x}\dot{\phi}\tan{\phi} \, , \\
		\dot{y}&=&\tan(\phi)\dot{x} = \Psi^3 \, .
	\end{eqnarray*}
	
	Since $F^1=0$ we get $\tilde{\Phi}^1_i=0$. We also have
	\begin{eqnarray*}
	\tilde{\Phi}^2_1=-\frac{1}{8} \dot{\phi} (\cos (2 \phi)-5) \sec ^2(\phi) \dot{x} \, , && \tilde{\Phi}^2_2=\frac{1}{8} \dot{\phi}^2 (\cos (2 \phi)-5) \sec ^2(\phi) \, , \\
	K^3_1=-\sec^2(\phi)\dot{x}\, , && K^3_2= \sec^2(\phi)\dot{\phi} \, .
	\end{eqnarray*}

	The components of the Massa-Pagani connection given in (\ref{conn cmpts}) are
	
	\begin{eqnarray*}
		\tilde{\nabla}_{\tilde{\Gamma}} \tilde{H}_1 = \tilde{\Gamma}^2_1\tilde{H}_2 =  -\dot{x}\tan{\phi} \tilde{H}_2\, , && \tilde{\nabla}_{\tilde{\Gamma}} \tV_1 = -\dot{x}\tan{\phi}  \tV_2\, , \\
		\tilde{\nabla}_{\tilde{\Gamma}} \tilde{H}_2 = \tilde{\Gamma}^2_2\tilde{H}_2 =  -\dot{\phi}\tan{\phi}  \tilde{H}_2\, , && \tilde{\nabla}_{\tilde{\Gamma}} \tV_2 = -\dot{\phi}\tan{\phi}  \tV_2\, , \\
		\tilde{\nabla}_{\tilde{H}_1} \tilde{H}_2 = \tilde{\nabla}_{\tilde{H}_2} \tilde{H}_1 = -\tan{\phi}  \tilde{H}_2\, , && \tilde{\nabla}_{\tilde{H}_1} \tV_2 = \tilde{\nabla}_{\tilde{H}_2} \tV_1 = -\tan{\phi} \tV_2 \, , \\
		\tilde{\nabla}_{\tilde{\Gamma}}\frac{\partial}{\partial y} = 0 \, , &&
		\tilde{\nabla}_{\tilde{H}_a}\frac{\partial}{\partial y} = 0 \, , \\
		\tilde{\nabla}_{\frac{\partial}{\partial y}}\frac{\partial}{\partial y} = 0 \, , &&
	\end{eqnarray*}
	where
	$$
	\tV_1=\frac{\partial}{\partial \dot{\phi}}, \quad \tV_2=\frac{\partial}{\partial \dot{x}}, \quad \tilde{H}_1=\frac{\partial}{\partial \phi} - \tilde{\Gamma}^2_1\frac{\partial}{\partial \dot{x}} = \frac{\partial}{\partial \phi} + \dot{x}\tan{\phi} \frac{\partial}{\partial \dot{x}} ,
	\quad \tilde{H}_2= \frac{\partial}{\partial x} + \dot{\phi}\tan{\phi} \frac{\partial}{\partial \dot{x}} - \tan{\phi} \frac{\partial}{\partial y} \, .
	$$

	The shape map $A_{\Gamma}$ has eigenvalues $\mu_1=0$ with corresponding eigenspace
	$$
	E_1=\Sp\left\{\frac{\partial}{\partial y}, \dot{\phi}\tilde{H}_1+\dot{x}\tilde{H}_2, \tilde{\Gamma}\right\}
	$$
	and $\mu_{2,3}=\pm\sqrt{-\Phi^2_2}$ with corresponding eigenspaces
	\begin{eqnarray*}
		E_2&=&\Sp\left\{-K^3_2\frac{\partial}{\partial y} + \mu_2\tilde{H}_2 + \mu_2^2\tilde{V}_2\right\}\, ,\\
		E_3&=&\Sp\left\{-K^3_2\frac{\partial}{\partial y} + \mu_3\tilde{H}_2 + \mu_3^2\tilde{V}_2\right\}\, .
	\end{eqnarray*}

\end{ex}

\begin{ex}[Ball rolling on a spherical surface] Consider a ball of radius $r$ and mass $m$ that is rolling without sliding on the inner side of half a sphere of radius $R+r$. We can take coordinates $(x,y)$ for the centre of the ball, which moves on a half sphere $\Sigma$ of radius $R$, with $z=-\sqrt{R^2-x^2-y^2}+R$ and we can take Euler angles for the orientation of the ball, that is, as local coordinates for $SO(3)$ (see \cite[Section 5.3]{BalSan} and \cite{{NeiFuf72}}). The configuration space is $SO(3)\times\Sigma$.
	
	The Lagrangian of the system is
	\begin{eqnarray*}
		\bar{L}(x,y,\dot{x},\dot{y},w_1,w_2,w_3)&=&\frac{m}{2}\left(\frac{R^2-y^2}{R^2-x^2-y^2}\dot{x}^2 + \frac{2xy}{R^2-x^2-y^2}\dot{x}\dot{y}  + \frac{R^2-x^2}{R^2-x^2-y^2}\dot{y}^2\right) \\
		&& + \frac{I}{2}(w_1^2+w_2^2+w_3^2) - mg(R-\sqrt{R^2-x^2-y^2}) \, ,
	\end{eqnarray*}
	where $w=(w_1,w_2,w_3)$ is the angular velocity of the ball in space coordinates, $I=\frac{2}{5}mr^2$ and $g$ denotes gravity acceleration. The nonholonomic constraints of rolling without sliding are
	$$
	\dot{x}=-r(w_2n_3-w_3n_2),
	\quad
	\dot{y}=-r(w_3n_1-w_1n_3)\, ,
	$$
	where $n(x,y)=(n_1,n_2,n_3)$ is the outward normal unit vector of the half sphere on which the centre of the ball moves, so
	$$
	n_1=\frac{x}{R}, \quad n_2=\frac{y}{R} \quad \mbox{and} \quad n_3=\frac{-\sqrt{R^2-x^2-y^2}}{R}.
	$$
	
	In terms of Euler angles $0<\varphi,\psi<2\pi$, $0<\theta<\pi$ we have
	\begin{eqnarray*}
		w_1&=&\dot{\theta}\cos{\varphi} + \dot{\psi}\sin{\varphi}\sin{\theta}\, , \\
		w_2&=&\dot{\theta}\sin{\varphi} - \dot{\psi}\cos{\varphi}\sin{\theta}\, , \\
		w_3&=&\dot{\varphi} + \dot{\psi}\cos{\theta}\, .
	\end{eqnarray*}
	Therefore, using spherical coordinates $(\alpha,\beta)$ for the spherical surface and Euler angles $(\varphi,\psi,\theta)$ for the orientation of the ball we get the Lagrangian
	$$
	L(\varphi,\psi,\theta,\alpha,\beta,\dot{\varphi},\dot{\psi},\dot{\theta},\dot{\alpha},\dot{\beta})= \frac{mR}{2}(\dot{\alpha}^2+\dot{\beta}^2\sin^2\alpha)
	+ \frac{I}{2}(\dot{\psi}^2 + \dot{\theta}^2 + \dot{\varphi}^2 + 2\dot{\psi}\dot{\varphi}\cos\theta) -mg(R+R\cos\alpha)
	$$
	and the nonholonomic constraints
	\begin{eqnarray*}
		\dot{\alpha} &=& \frac{r}{R}\left( \dot{\theta}\sin(\beta-\varphi) + \dot{\psi}\sin\theta\cos(\varphi-\beta) \right) = \Psi^4 \, ,\\
		\dot{\beta} &=& -\frac{r}{R}\left( \dot{\varphi} + \dot{\psi}(\cos\theta + \cot\alpha\sin\theta\sin(\beta-\varphi)) - \dot{\theta}\cot\alpha\cos(\varphi-\beta) \right) = \Psi^5 \, .
	\end{eqnarray*}
	
	Now using the Lagrange-d'Alembert principle we obtain the nonholonomic equations
	\begin{align*}
		&mR\ddot{\alpha} - mR\dot{\beta}^2\sin\alpha\cos\alpha-mgR\sin\alpha = \lambda_4 \, , \\
		&2mR\dot{\alpha}\dot{\beta}\sin\alpha\cos\alpha + mR\ddot{\beta}\sin^2\alpha = \lambda_5 \, ,\\
		&I\ddot{\varphi} -I\dot{\theta}\dot{\psi}\sin\theta + I\ddot{\psi}\cos\theta = \frac{r}{R}\lambda_5 \, ,\\
		&I\ddot{\psi} - I\dot{\theta}\dot{\varphi}\sin\theta + I\ddot{\varphi}\cos\theta = \frac{-r}{R}(\sin\theta\cos(\beta-\varphi)\lambda_4 - (\cos\theta + \cot\alpha\sin\theta\sin(\beta-\varphi))\lambda_5) \, ,\\
		&I\ddot{\theta} + I\dot{\psi}\dot{\varphi}\sin\theta = \frac{-r}{R}(\sin(\beta-\varphi)\lambda_4 + \cos(\beta-\varphi)\cot\alpha\lambda_5) \, ,\\
		&\dot{\alpha} = \frac{r}{R}\left( \dot{\theta}\sin(\beta-\varphi) + \dot{\psi}\sin\theta\cos(\varphi-\beta) \right) \, ,\\
		&\dot{\beta} = -\frac{r}{R}\left( \dot{\varphi} + \dot{\psi}(\cos\theta + \cot\alpha\sin\theta\sin(\beta-\varphi)) - \dot{\theta}\cot\alpha\cos(\varphi-\beta) \right) \, .
	\end{align*}
	
	We use the constraints to replace $\ddot{\alpha}$, $\ddot{\beta}$, $\dot{\alpha}$ and $\dot{\beta}$ in the first two equations and get the Lagrange multipliers $\lambda_4$ and $\lambda_5$ in terms of $(\varphi,\psi,\theta,\alpha,\beta,\dot{\varphi},\dot{\psi},\dot{\theta},\ddot{\varphi},\ddot{\psi},\ddot{\theta})$. Then we plug these expressions into the next three equations and solve for $\ddot{\varphi}$, $\ddot{\psi}$ and $\ddot{\theta}$ to get the constrained SODE
	\begin{eqnarray*}
		\ddot{\varphi} &=& F^1(\varphi,\psi,\theta,\alpha,\beta,\dot{\varphi},\dot{\psi},\dot{\theta}) \, ,\\
		\ddot{\psi} &=& F^2(\varphi,\psi,\theta,\alpha,\beta,\dot{\varphi},\dot{\psi},\dot{\theta}) \, ,\\
		\ddot{\theta} &=& F^3(\varphi,\psi,\theta,\alpha,\beta,\dot{\varphi},\dot{\psi},\dot{\theta}) \, ,\\
		\dot{\alpha} &=& \frac{r}{R}\left( \dot{\theta}\sin(\beta-\varphi) + \dot{\psi}\sin\theta\cos(\varphi-\beta) \right) \, ,\\
		\dot{\beta} &=& -\frac{r}{R}\left( \dot{\varphi} + \dot{\psi}(\cos\theta + \cot\alpha\sin\theta\sin(\beta-\varphi)) - \dot{\theta}\cot\alpha\cos(\varphi-\beta) \right) \, .
	\end{eqnarray*}	
	The expressions for $F^1,F^2$ and $F^3$ can be computed with Mathematica but are too involved to include them here.
	
	Now we can compute the nonzero coefficients of the Massa-Pagani connection. On one side, if we take the standard metric on $\Sigma$, $R^2d\alpha^2+R^2\sin^2\alpha d\beta^2$, and the corresponding Riemannian connection, we have that the only nonvanishing coefficients are $\Upsilon^4_{55}=-\sin\alpha\cos\alpha$ and $\Upsilon^5_{45}=\frac{\cos\alpha}{\sin\alpha}$, so we get the corresponding nonvanishing coefficients for the Massa-Pagani connection:
	$$
	\tilde{\nabla}_{\frac{\partial}{\partial\beta}}\frac{\partial}{\partial\beta} = -\sin\alpha\cos\alpha\frac{\partial}{\partial\alpha}, \quad
	\tilde{\nabla}_{\frac{\partial}{\partial\alpha}}\frac{\partial}{\partial\beta}=\tilde{\nabla}_{\frac{\partial}{\partial\beta}}\frac{\partial}{\partial\alpha}=\frac{\cos\alpha}{\sin\alpha}\frac{\partial}{\partial\beta}\, .
	$$
	The other components in (\ref{conn cmpts}) are
	\begin{eqnarray*}
		\tilde{\nabla}_{\tGamma}\frac{\partial}{\partial\alpha}&=& \frac{r}{R}\csc^2\alpha\left(\dot{\psi}\sin\theta\sin(\varphi-\beta)+\dot{\theta}\cos(\varphi-\beta)\right)\frac{\partial}{\partial \beta} \, ,\\
		\tilde{\nabla}_{\tGamma}\frac{\partial}{\partial\beta}&=&\frac{-r}{R}\left(\dot{\theta}\cos(\varphi-\beta)+\dot{\psi}\sin\theta\sin(\varphi-\beta)\right)\frac{\partial}{\partial \alpha} \\
		&& + \frac{r}{R}\left( \dot{\theta}\cot\alpha\sin(\beta-\varphi)+\dot{\psi}\cot\alpha\sin\theta\cos(\varphi-\beta)\right)\frac{\partial}{\partial \beta} \, ,\\
		\tilde{\nabla}_{\tH_1}\frac{\partial}{\partial\alpha}&=&0 \, , \quad
		\tilde{\nabla}_{\tH_1}\frac{\partial}{\partial\beta}=0 \, ,\\
		\tilde{\nabla}_{\tH_2}\frac{\partial}{\partial\alpha}&=&-\frac{r}{R}\csc^2\alpha\sin\theta\sin(\varphi-\beta)\frac{\partial}{\partial\beta} \, , \\
		\tilde{\nabla}_{\tH_2}\frac{\partial}{\partial\beta}&=&\frac{r}{R}\sin\theta\sin(\varphi-\beta)\frac{\partial}{\partial\alpha}-\frac{r}{R}\cot\alpha\sin\theta\cos(\varphi-\beta)\frac{\partial}{\partial\beta} \, ,\\
		\tilde{\nabla}_{\tH_3}\frac{\partial}{\partial\alpha}&=&-\frac{r}{R}\csc^2\alpha\cos(\varphi-\beta)\frac{\partial}{\partial\beta} \, ,\\
		\tilde{\nabla}_{\tH_3}\frac{\partial}{\partial\beta}&=&\frac{r}{R}\cos(\varphi-\beta)\frac{\partial}{\partial\alpha}-\frac{r}{R}\cot\alpha\sin(\beta-\varphi)\frac{\partial}{\partial\beta} \, ,
	\end{eqnarray*}
	and $\tnabla_{\tGamma} \tH_a$, $\tnabla_{\tGamma} \tV_a$, $\tnabla_{\tH_a} \tH_b$, $\tnabla_{\tH_a} \tV_b$, which depend on derivatives of $F^1$, $F^2$ and $F^3$.

\end{ex}

\section*{Acknowledgements}
GP and MFP thank Instituto de Ciencias Matem\'aticas (ICMAT) for its warm hospitality. MFP acknowledges financial support from the FWO (Research Foundation - Flanders). DMdD acknowledges financial support from the Spanish Ministry of Science and Innovation, under grants PID2019-106715GB-C21 and the Spanish National Research Council, through the ``Ayuda extraordinaria a Centros de Excelencia Severo Ochoa” R\&D (CEX2019-000904-S).

\medskip

G.E.\,Prince \\
Department of Mathematics and Statistics, La Trobe University,\\
Victoria 3086, Australia \\
Email: \url{g.prince@latrobe.edu.au}

M.\,Farr\'{e} Puiggal\'{i}\\
University of Antwerp, Department of Mathematics,\\
Middelheimlaan 1, 2020 Antwerpen, Belgium \\
Email: \url{marta.farrepuiggali@uantwerpen.be}

D.J.\,Saunders \\
Lepage Research Institute,\\
17. novembra 1, 081 16 Pre\v{s}ov, Slovakia \\
Email: \url{david@symplectic.email}

D.\,Mart\'{i}n de Diego\\
Instituto de Ciencias Matem\'aticas (CSIC-UAM-UC3M-UCM),\\
C/Nicol\'as Cabrera 13-15, 28049 Madrid, Spain\\
Email: \url{david.martin@icmat.es}

\end{document}